\newif\iffullversion

\fullversiontrue

\documentclass[a4paper,USenglish,nolineno]{socg-lipics-v2019}
\usepackage{amsthm}
\usepackage{amsmath}
\usepackage{mathbbol}
\usepackage{amsfonts}
\usepackage{amssymb}
\usepackage{graphicx}
\usepackage{algorithm}
\usepackage[capitalise]{cleveref}
\usepackage{url}
\usepackage{epsfig}

\usepackage{blkarray}

\newcommand{\R}{\mathbb{R}}
\newcommand{\eps}{\epsilon}

\newcommand{\cost}{\mathrm{cost}}
\newcommand{\field}{\mathbb{K}}
\newcommand{\grade}{\mathrm{gr}}
\newcommand{\col}{\mathrm{col}}
\newcommand{\pres}{P}
\newcommand{\row}{\mathrm{row}}
\newcommand{\spann}{\mathrm{span}}
\newcommand{\push}{\mathrm{push}}

\newcommand{\domain}{\Omega}
\newcommand{\onehalforone}{c}
\newcommand{\region}{R}
\newcommand{\ring}{\mathcal{R}}
\newcommand{\face}{f}

\newcommand{\dmatch}{d_{\mathrm{match}}}

\newcommand{\numgenerators}{k}

\iffullversion
\newtheorem*{example*}{Example}
\fi

\usepackage{color}

\title{Exact Computation of the Matching Distance on 2-Parameter Persistence Modules}
\author{Michael Kerber}{Graz University of Technology, Graz, Austria}{kerber@tugraz.at}{https://orcid.org/0000-0002-8030-9299}{Supported by Austrian Science Fund (FWF) grant number P 29984-N35}
\author{Michael Lesnick}{University at Albany, SUNY, United States}{mlesnick@albany.edu}{https://orcid.org/0000-0003-1924-3283}{}
\author{Steve Oudot}{Inria Saclay -- \^Ile-de-France, Palaiseau, France}{steve.oudot@inria.fr}{https://orcid.org/0000-0003-2939-9417}{}

\authorrunning{M. Kerber, M. Lesnick and S. Oudot}

\Copyright{Michael Kerber, Michael Lesnick and Steve Oudot}

\acknowledgements{
This work was initiated at the BIRS workshop ``Multiparameter Persistent Homology'' (18w55140) in Oaxaca, Mexico (Aug. 2018). We thank Jan Reininghaus 
and the other members of the discussion group on this topic for \mbox{fruitful initial exchanges}.  We thank Matthew Wright for helpful discussions about line arrangements, slices, and the computational aspects of 2-parameter persistence.}



\EventEditors{Gill Barequet and Yusu Wang}
\EventNoEds{2}
\EventLongTitle{35th International Symposium on Computational Geometry (SoCG 2019)}
\EventShortTitle{SoCG 2019}
\EventAcronym{SoCG}
\EventYear{2019}
\EventDate{June 18--21, 2019}
\EventLocation{Portland, United~States}
\EventLogo{socg-logo}
\SeriesVolume{129}
\ArticleNo{46} 

\iffullversion
\hideLIPIcs
\else
\relatedversion{A full version of this paper is available at \url{https://arxiv.org/abs/1812.09085}}
\fi

\begin{document}

\maketitle
\begin{abstract}
The matching distance is a pseudometric on multi-parameter persistence modules, defined in terms of the 
weighted bottleneck distance on the restriction of the modules to affine lines.  It is known that this distance is stable in a reasonable sense, and can be efficiently approximated, which makes it a promising tool for practical applications.  In this work, we show that in the 2-parameter setting, the matching distance can be computed 
exactly in polynomial time.  Our approach subdivides the space of affine lines into regions, via a line arrangement. 
In each region, the matching distance restricts to a simple 
analytic function, whose maximum is easily computed.  As a byproduct, our analysis establishes that the matching distance
is a rational number, if the bigrades of the input modules are rational.
\end{abstract}

\section{Introduction}\label{Sec:Intro}
Multi-parameter persistent homology is receiving growing
attention, both from the theoretical and computational points of view.
Its motivation lies in the possibility of extending the success
of topological data analysis to settings where the structure of data is best captured by 2-parameter rather than 1-parameter constructions.  
The basic algebraic objects of study in multi-parameter persistence are certain commutative diagrams of vector spaces called \emph{persistence modules}.  
In the 1-parameter setting, persistence modules decompose in an essentially unique way into simple summands called \emph{interval modules}.  The decomposition is specified by a discrete invariant called a \emph{persistence diagram}.  In contrast, the algebraic structure of a 2-parameter persistence module (henceforth, \emph{bipersistence module}) can be far more complex.  As a result, a good definition of persistence diagram is unavailable for bipersistence modules~\cite{Carlsson2009}. 

Nevertheless, it is still possible
to define meaningful notions of distance between 
 multi-parameter persistence modules.  Distances on 1-parameter persistence modules play an essential role in both theory and applications.  To extend such theory and applications to the multi-parameter setting, one needs to select a suitable distance on multi-parameter persistence modules.  However, progress on finding well-behaved, efficiently computable distances on multi-parameter persistence modules has been slow, and this is been an impediment to progress in practical applications.

The most widely studied and applied distances in the 1-parameter setting are the \emph{bottleneck distance} and the \emph{Wasserstein distance}~\cite{EdelsbrunnerHarer2010,Oudot2015}.
Both can be efficiently computed via publicly available code~\cite{kmn-geometry}. In the multi-parameter setting, the distance that has received the most attention is the \emph{interleaving distance}. On 1-parameter persistence modules, the interleaving and bottleneck distances are equal~\cite{Lesnick2015}.  The interleaving distance is theoretically well-behaved; in particular, on modules over prime fields, it is
the most discriminative distance among the ones satisfying a certain stability condition~\cite{Lesnick2015}.  However, it was proven recently
to be NP-hard to compute or even approximate to any constant factor less than~three on bipersistence modules, even under the restriction that the modules decompose as direct sums of {\em interval modules}~\cite{Bjerkevik}.
While the problem lies in $P$ when further restricting the focus to the sub-class of interval modules themselves~\cite{dx-computing}, in practice one often encounters modules that are not interval modules (nor even direct sums thereof).

This motivates the search for a more computable surrogate for the
interleaving distance.  One natural candidate is the \emph{matching distance}, introduced by
Cerri et al.  \cite{cerri2013betti}. It is a lower bound for the
interleaving distance; this is implicit in \cite{cerri2013betti} and
shown explicitly in \cite{Landi2018}.  In the 2-parameter setting, the
matching distance is defined as follows: Given a pair of bipersistence
modules, we call an affine line $\ell$ in parameter space with
positive slope a \emph{slice}.  Restricting the modules to $\ell$
yields a pair of 1-parameter persistence modules, which we
call \emph{slice modules}.  These slice modules have a well-defined
bottleneck distance, which we multiply by a positive weight 
to ensure that the matching distance will be a lower bound for the
interleaving distance.  The matching distance is defined as the
supremum of these weighted bottleneck distances over all slices.
See \cref{sec:matching_distance} for the precise definition.  The
definition generalizes readily to $n$-parameter persistence modules,
for any $n\geq 1$; when $n=1$, the matching distance is equal to the
bottleneck distance.

As Cerri et al. have observed, to approximate the matching distance up to any (absolute) precision, it suffices to sample the space of slices
sufficiently densely and to return the maximum weighted bottleneck distance
encountered.  For a constant number of scale parameters
and approximation quality $\eps$,
a polynomial number of slices are sufficient in terms of module size and $\frac{1}{\eps}$, yielding a polynomial
time approximation algorithm.~\cite{BiasottiCerriFrosini2011}.
This approach has  been recently applied to the virtual ligand screening problem in computational chemistry~\cite{Keller2018}.
To the best of our knowledge, there is no other previous work in which the problem of computing the matching distance has been considered.

\subparagraph{Our contribution}
We give an algorithm that computes the exact matching distance
between a pair of bipersistence modules in time polynomial
with respect to the size of the input.  We assume that each persistence module is specified by a \emph{presentation}.  Concretely, this means that the module is specified by a matrix encoding the generators and relations, 
with each row and each column labeled by a point in $\R^2$; see \cref{sec:modules}.

To explain our strategy for computing the matching distance, consider the function $F$
that assigns a slice to its weighted bottleneck distance. 
The matching distance is then simply the supremum of $F$,
taken over all slices.  $F$ has a rather complicated structure, since it depends
on the longest edge of a perfect matching in a bipartite graph
whose edges lengths depend on both the slice and the two modules given as input.
When the slice changes, the matching realizing the bottleneck distance
undergoes combinatorial changes, making the function $F$ difficult to treat
analytically.

We show, however, that the space of slices can be divided into 
polynomially many regions so that the restriction of $F$ to each region
takes a simple closed from.  
Perhaps surprisingly, if we parameterize the space of slices by the right half plane $\domain\subset \R^2$, 
the boundary between these regions can be expressed by the union
of polynomially many lines in $\domain$, 
making each region convex and bounded by (possibly unbounded) line segments.  (This is analogous to the observation of \cite{Lesnick2015b} that for a single persistence module, the locus of lines where the combinatorial structure underlying the slice module can change is described by a line arrangement.)  
Moreover, the restriction of $F$ to each cell attains its supremum
at a boundary vertex of the cell, or as the limit of an unbounded line segment in $\Omega$; this follows from a straightforward case analysis.  These observations together lead to a simple
polynomial time algorithm to compute the matching distance.

The characterization of the matching distance underlying our algorithm also makes clear that if the row and column labels of the presentations of the input modules
have rational coordinates, then the matching distance
is rational as well. We are not aware of a simpler argument
for this property.

\subparagraph{Outline}
We introduce the underlying topological concepts in \cref{sec:modules},
and introduce the matching distance in \cref{sec:matching_distance}.
We define the line arrangement subdividing the slice space
in \cref{sec:arrangement}
and give the algorithm to maximize each cell of the arrangement
in \cref{sec:maximization}.
We conclude in \cref{sec:discussion}.

\section{Persistence modules}
\label{sec:modules}

\subparagraph{Single-parameter modules}
Let $\field$ be a fixed finite field throughout. 
A \emph{persistence module} $M$ over $\R$ is an assignment of $\field$-vector spaces 
$M_x$ to real numbers $x$, and linear maps $M_{x\to y}:M_x\to M_{y}$ to a pair of
real numbers $x\leq y$, such that $M_{x\to x}$ is the identity 
and $M_{y\to z}\circ M_{x\to y}=M_{x\to z}$.
Equivalently, in categorical terms, a persistence module is a functor
from $\R$ (considered as a poset category) to the category of $\field$-vector spaces.

A common way to arrive at a persistence module
is to consider a nested sequence 
\[X_1\subseteq X_2\subseteq \ldots \subseteq X_n\]
of simplicial complexes 
and to apply homology with respect to a fixed dimension and base field
$\field$. This yields a sequence
\[H_p(X_1,\field)\to H_p(X_2,\field)\to \ldots \to H_p(X_n,\field)\]
of vector spaces and linear maps. To obtain a persistence module over $\R$, 
we pick \emph{grades} $s_1 < s_2 < \ldots < s_n$
and set $M_x:=0$ if $x<s_1$ and $M_x:=H_p(X_i,\field)$ with $i=\max\{j\mid s_j\leq x\}$ otherwise.
For $y\geq x$ and $M_y=H_p(X_j,\field)$, we define $M_{x\to y}$ as
the map $H_p(X_i,\field)\to H_p(X_j,\field)$ induced by the inclusion map $X_i\to X_j$.

\subparagraph{Finite presentations}
In this work, we restrict our attention to persistence modules 
that are finitely presented in the following sense.
A \emph{finite presentation} is an $\numgenerators\times m$ matrix $\pres$ over $\field$,
where each row and each column is labeled by a number in $\R$, called the \emph{grade}, such that if $\pres_{ij}\ne 0$, then $\grade(\row_i)\leq \grade(\col_j)$; here $\grade(-)$ denotes the grade of a row or column.    We refer to the multiset of grades of all rows and columns of $\pres$ simply as the \emph{set of grades of }$\pres$, and denote this set as $\grade(\pres)$.
A finite presentation gives rise to a persistence module, as we describe next.  The rows of $\pres$ represent the generators of the module, while the columns of $\pres$ encode
relations (or syzygies) on the generators. Concretely, 
let $e_1,\ldots,e_\numgenerators$ denote the standard basis of $\field^\numgenerators$, 
and for $x\in\R$ define the subset
\[
\mathrm{Gen}_x:=\{e_i\mid \grade(\row_i)\leq x\}
\]
Likewise, define
\[
\mathrm{Rel}_x:=\{\col_j\mid \grade(\col_j)\leq x\}
\]
Then, we define
\[
M^\pres_x:=\spann(\mathrm{Gen}_x) / \spann(\mathrm{Rel}_x)
\]
and $M^\pres_{x\to y}$ simply as the map induced by the inclusion map $\spann(\mathrm{Gen}_x)\to\spann(\mathrm{Gen}_y)$.   
It can be checked easily that this indeed defines a persistence module $M^\pres$.  If a persistence module $N$ is isomorphic to $M^\pres$, we say that $\pres$ is a \emph{presentation of $N$.}
We call a persistence module $N$ \emph{finitely presented} if there exists a finite presentation of $N$. For instance, persistence modules as above arising from a finite
simplicial filtration are finitely presented. 
Also, the representation theorem of persistence~\cite{ZomorodianCarlsson2005,Corbet2018}
states that the category of persistence modules over $\field$ is isomorphic to the category of graded $\ring$-modules
with an appropriately chosen ring $\ring$.
With that, a persistence module is finitely presented if and only if the corresponding $\ring$-module
is finitely presented (in the sense of a module).

We assume for concreteness that a finite presentation is encoded 
in terms of a sparse matrix representation~\cite{EdelsbrunnerHarer2010}; the size of the presentation is understood to be the bitsize of this sparse matrix.   
In what follows, for finite a presentation with $\numgenerators$ generators
and $m$ relations, we will express complexity bounds in terms
of $n:=\numgenerators+m$, the number of generators and relations.
Note that an algorithm polynomial in $n$ is also polynomial in the 
size of the presentation.

\subparagraph{Persistence diagrams}
A \emph{persistence diagram} 
is a finite multi-set of points of the form $(b,d)\in \R\times (\R\cup\{\infty\})$ with $b<d$.  
A well-known structure theorem tells us that we can associate to any finitely presented persistence module $M$ a persistence diagram $D(M)$, and that this determines $M$ up to isomorphism \cite{crawley2015decomposition}. 

Given a presentation of $M$, the persistence diagram can be computed by bringing the presentation matrix
into echelon form. This process takes cubic time
in $n$ 
using Gaussian elimination~\cite{EdelsbrunnerHarer2010,ZomorodianCarlsson2005}, 
or $O(n^\omega)$ time using fast matrix multiplication, 
where $\omega\leq 2.373$~\cite{mms-zigzag}.  
\iffullversion
Elimination-based approaches to computing persistent homology perform very well in practice, and are routinely used to study real data.
\fi

\begin{lemma}\label{Lem:Pres_Grades_And_Persistence_Diagram}
For $P$ a finite presentation of a persistence module $M$ and $(b,d)\in D(M)$,
\begin{enumerate}
\item $b$ is the grade of some row of $P$, 
\item if $d<\infty$, then $d$ is the grade of some column of $P$.
\end{enumerate}
\end{lemma}

\begin{proof}
  This follows from the correctness of the basic matrix reduction algorithm for computing persistent homology, as described in \cite{ZomorodianCarlsson2005}.
\end{proof}

\subparagraph{Bottleneck distance}
Consider two persistence diagrams $D_1$ and $D_2$ and a bijection $\sigma:D'_1\to D'_2$ for $D'_1\subseteq D_1$ and $D'_2\subseteq D_2$.  For $\delta>0$, we define
$\cost(\sigma):=\max(A,B)$, where 
\begin{align*}
A&=\max\, \{\max(|a-c|,|b-d|) \mid (a,b)\in D_1',\ \sigma(a,b)=(c,d)\},\\
B&=\max\, \{(b-a)/2\mid (a,b)\in (D_1\setminus D'_1)\cup (D_2\setminus D'_2)\},
\end{align*}
and it is understood that $\infty-\infty=0$.  
We define the bottleneck distance $d_B$ by 
\[d_B(D_1,D_2)=\min \{\epsilon\mid \textup{there exists a matching of cost $\epsilon$ between $D_1$ and $D_2$}\}.\]
For persistence modules $M$ and $N$, we write $d_B(D(M),D(N))$ simply as $d_B(M,N)$.

\begin{lemma}
\label{lem:bottelneck_structure}
Let $\pres^M$ and $\pres^N$ be finite presentations of persistence modules $M$ and $N$, respectively.  $d_B(M,N)$ is equal to one of the following:
\begin{enumerate}
\item The difference of a grade of $\pres^M$ and a grade of $\pres^N$,
\item half the difference of two grades in $\pres^M$,
\item half the difference of two grades in $\pres^N$.
\end{enumerate}
\end{lemma}

\begin{proof}
This follows immediately from Lemma \ref{Lem:Pres_Grades_And_Persistence_Diagram} and the definition of $d_B$.
\end{proof}

Given two finite persistence diagrams $D$, $D'$
with $m$ points each,
we can compute $d_B(D,D')$ in time 
to
 $O(m^{1.5}\log m)$~\cite{eik-geometry}; 
see~\cite{kmn-geometry} for details, 
including a report on practical efficiency.
The number of points in a diagram
is at most the number of generators of the corresponding module,
and hence upper-bounded by $n$.
Thus, the complexity of computing the bottleneck distance
of two persistence modules is dominated by the computation
of the persistence diagrams, and has worst-case complexity $O(n^\omega)$.

\subparagraph{Bipersistence modules}
The definitions of persistence modules and presentations extend to higher
dimensions without difficulty.  In the 2-parameter setting, this goes as follows: 
Define a partial
order~$\leq$ on $\R^2$ by $p\leq q$ if $p_x\leq q_x$ and $p_y\leq q_y$.  A
\emph{bipersistence module} is an assignment of $\field$-vector spaces 
$M_p$ to points $p\in \R^2$, and linear maps $M_{p\to q}:M_p\to M_{q}$ to pairs of
points $p\leq q\in \R^2$, such that $M_{p\to p}$ is the identity 
and $M_{q\to r}\circ M_{p\to q}=M_{p\to r}$ whenever $p\leq
q\leq r$.
In topological data analysis, $2$-dimensional persistence
modules typically arise by applying homology to a bifiltered 
simplicial complex 

A \emph{finite presentation} of a bipersistence module $\R^2$ is defined in the same way as for one-parameter persistence modules, except
that the labels of each row/column are elements of $\R^2$, and the $\leq$ relation appearing in the definition means the partial order over $\R^2$.
From now on, we will assume that all bipersistence modules considered are finitely presented.

\iffullversion
\begin{example*}
Let $M$ be the bipersistence module given by \[M_{p}=\begin{cases}\field &\textup{ if $p\in [0,1)\times [0,1)$}\\ 0 &\textup{ otherwise,}\end{cases}\qquad M_{p\to q}=\begin{cases}\mathrm{Id}_\field &\textup{ if $p,q\in [0,1)\times [0,1)$}\\ 0 &\textup{otherwise.}\end{cases}.\]
Then a presentation of $M$ is given by
\[
\begin{blockarray}{ccc}
 & (1,0) & (0,1)  \\
\begin{block}{c[cc]}
  (0,0) & 1 & 1 \\
\end{block}
\end{blockarray}. \]
\end{example*}
\fi

In topological data analysis, we do not typically have immediate access to a presentation of a bipersistence module $M$, but rather to a chain complex of bipersistence modules for which $M$ is a homology module.  However, it has recently been observed that from such a chain complex, a (minimal) presentation of $M$ can be computed in cubic time~\cite{lesnick2018computing}.
The algorithm for this is practical, and has been implemented in the software package RIVET~\cite{RIVET}.

\section{The matching distance}
\label{sec:matching_distance}

\subparagraph{Slices}
We define a \emph{slice} as a line $\ell:y=sx+t$ where $s$ and $t$ are
real numbers with $s>0$.  Let $\lambda:\R\to\ell$ be an order-preserving, isometric parameterization of the slice.  Concretely, such a parameterization is given by
$\lambda(x) = \frac{1}{\sqrt{1+s^2}}(x,sx+t)$.
Given a bipersistence module $M$ and slice $\ell$, we 
define a (1-parameter) persistence module $M^\ell$ via $M^\ell_x:=M_{\lambda(x)}$, with its linear maps induced by $M$.  We call $M^\ell$ a \emph{slice module}.
\subparagraph{Matching distance}
For a slice $\ell:y=sx+t$, we define a weight
\[
w(\ell):=\begin{cases}
\frac{1}{\sqrt{1+s^2}} & s\geq 1\\
\frac{1}{\sqrt{1+\frac{1}{s^2}}} & 0<s< 1
\end{cases}
\]
$w(\ell)$ is maximized for slices with slope $1$,
and gets smaller when the slope goes to $0$ or to $\infty$.

Let $\Lambda$ denote the set of all slices.  
Writing $\domain:= (0,\infty)\times \R$, the map
$\alpha:\domain\to \Lambda$ that sends $(s,t)$ to the line $y=sx+t$
is clearly a bijective parameterization of $\Lambda$.
We equip $\Lambda$ with the topology induced by $\alpha$, 
using the subset topology $\Omega\subset\R^2$.

For two persistence modules $M$, $N$ over $\R^2$, 
we define a function $F^{M,N}:\Lambda\to [0,\infty)$ by
\[F^{M,N}(\ell):=w(\ell)\cdot d_B(M^\ell,N^\ell).\]
and we define the \emph{matching distance} between $M$ and $N$ as
 $\dmatch(M,N):=\sup F^{M,N}$.
As noted in the introduction, the weights $w(\ell)$ are chosen to ensure that $\dmatch$ is a lower bound for the interleaving distance. 

\begin{lemma}
\label{lem:F_cont}
Given two bipersistence modules $M$, $N$, the map $F^{M,N}$ is continuous.
\end{lemma}
\begin{proof}
$w$ is clearly continuous, so it suffices to show that the function $\ell\mapsto d_B(M^\ell,N^\ell)$ is continuous.
Let $\mathcal{D}$ denote the metric space of all finite persistence diagrams, with metric the bottleneck distance.   It follows from \cite[Theorem 2]{Landi2018} that the map sending a slice $\ell$ to the persistence diagram $D(M^\ell)$ is continuous with respect to the topology on $\mathcal{D}$.  Thus the map sending $\ell$ to the pair $(D(M^\ell),D(N^\ell))$ is also continuous.  Moreover, the bottleneck distance is clearly continuous as a map $\mathcal{D}\times \mathcal{D}\to [0,\infty)$, thanks to the triangle inequality.  
Thus, $\ell\mapsto d_B(M^\ell,N^\ell)$ is a composition
of continuous functions, hence continuous.
\end{proof}

\section{The arrangement}
\label{sec:arrangement}

In what follows, we fix two bipersistence modules $M$, $N$ 
(each with at most $n$ generators and relations)
and write $F:=F^{M,N}$.
Using the parameterization $\alpha$ from above, we also have a function
$F\circ\alpha:\Omega\to [0,\infty)$; slightly abusing notation,
we will also denote this map by $F$.

In this section, we construct a line arrangement in $\domain$ in such a way that it is simple to compute $\sup F$ on each face.  Recall that a \emph{line arrangement} of $\domain$ is the subdivision of $\domain$ into vertices, edges,
and faces induced by a finite set of distinct lines $L_1,\ldots,L_n$. 
The vertices of the arrangement are the intersection points of (at least) two lines,
the edges are maximal connected subsets of lines not containing any vertex,
and the faces are the connected components of $\domain\setminus\bigcup_{i=1}^n L_i$.
Clearly, each vertex, edge, and face of the arrangement is a convex set.
The boundary of each face consists of a finite number of edges and vertices.

\subparagraph{A first line arrangement}
For $v\in \R^2$, let $L_v$ denote the line $y = -v_x\, x + v_y$.  Note that $L_v\cap\domain$ is exactly the set of parameterizations of slices containing $v$.
Now, fix presentations $\pres^M$ and $\pres^N$ of $M$ and $N$.  Let $\mathcal{A}_0$ denote the arrangement in $\domain$ induced by the set of lines 
\[ \{L_v \mid v\in \grade(\pres^M) \cup \grade(\pres^N)\}.\]

In what follows, we will refine $\mathcal{A}_0$ by adding more lines into the arrangement.  For this we first need to introduce some definitions.

\subparagraph{Pushes}
For a point $p=(p_x,p_y)\in\R^2$ and a slice
$\ell:y=sx+t$, we define the \emph{push of $p$ onto $\ell$} as
\[\push(p,\ell):=\begin{cases} (p_x,s\cdot p_x + t) & \text{if $p$ lies below $\ell$}\\
(\frac{p_y-t}{s},p_y) & \text{if $p$ lies on or above $\ell$}\end{cases}
\]
Geometrically, $\push(p,\ell)$ gives the intersection point of $\ell$ and a vertical upward ray emanating
from $p$ in the first case, and the intersection of $\ell$ with a horizontal right ray emanating
from $p$ in the second case. See Figure~\ref{fig:pushes} for an illustration.

\begin{figure}

\centering
\includegraphics[width=5cm]{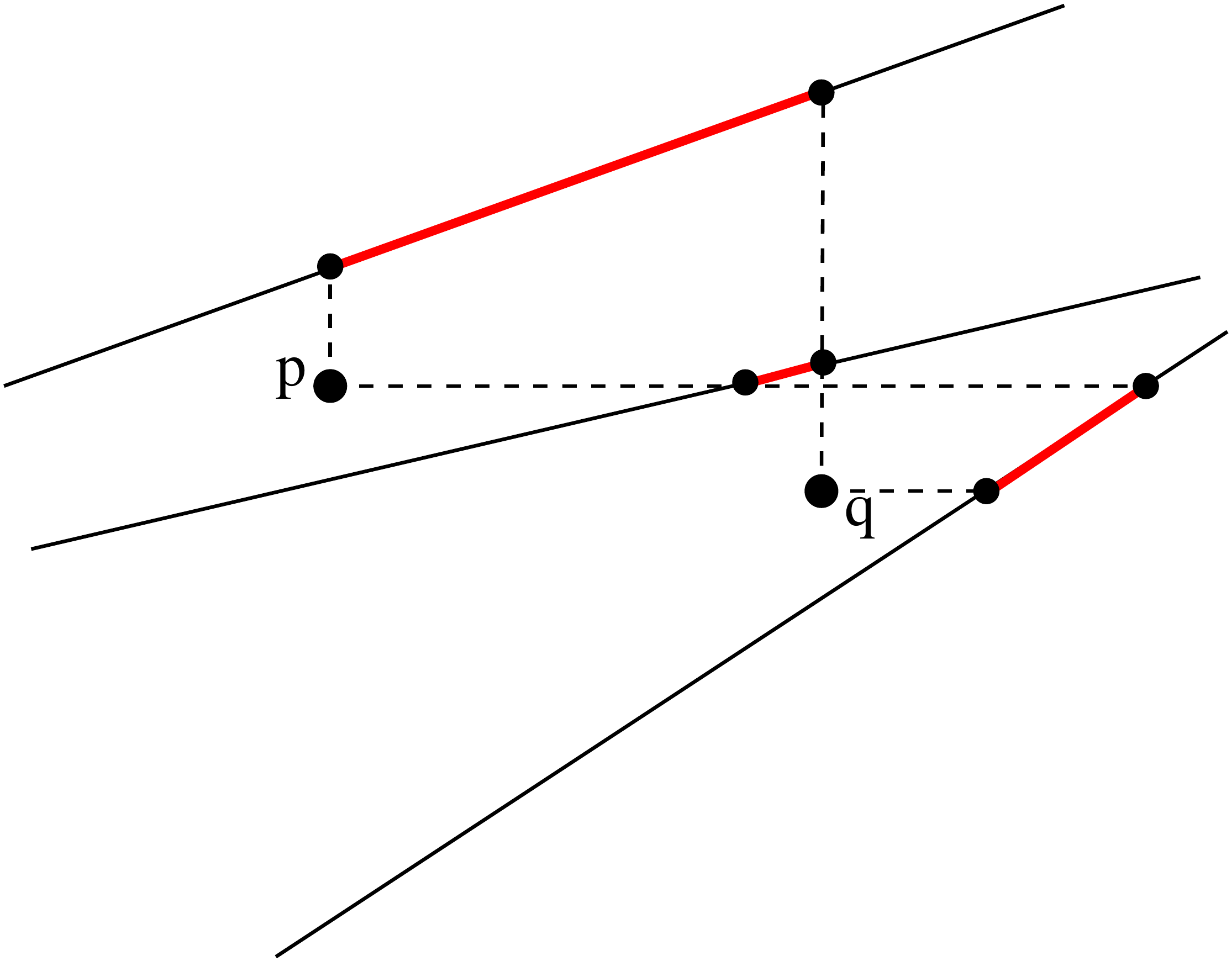}
\caption{The pushes of two points $p$ and $q$ to three different slices.
The length of the thick (red) line corresponds to the $\delta_{p,q}$ value
of the corresponding slice.}
\label{fig:pushes}
\end{figure}

\begin{remark}\label{Rem:Push_Pres}
A finite presentation of $M$ induces a finite presentation of $M^\ell$ with the same underlying matrix, and each row or column grade $p\in \R^2$ replaced with $\lambda^{-1}(\push(p,\ell))$.  
Clearly, this presentation can be obtained in 
time linear in the size of the presentation of~$M$.
\end{remark}

For $p,q\in \R^2$, define $\delta_{p,q}:\domain\to [0,\infty)$ by
\[
\delta_{p,q}(s,t):=\|\push(p,\ell)-\push(q,\ell)\|_2
= |\lambda^{-1}\circ\push(p,\ell) - \lambda^{-1}\circ\push(q,\ell)|
\]
with $\ell$ the slice defined by $s$ and $t$. 
Again, see Figure~\ref{fig:pushes} for an illustration.

We now give piecewise analytic formulae for $\delta_{p,q}$, which depend on whether the slice $\ell$ is below or above $p$ and $q$.

\begin{description}
\item [(I) slice is below both $p$ and $q$:]
\[\delta_{p,q}(s,t)=\|(\frac{p_y-t}{s},p_y)-(\frac{q_y-t}{s},q_y)\|_2
=\sqrt{\left(\frac{p_y-q_y}{s}\right)^2 + (p_y-q_y)^2}=|p_y-q_y|\sqrt{1+\frac{1}{s^2}}.
\]

\item [(II) slice is above both $p$ and $q$:]
\[\delta_{p,q}(s,t)=\|(p_x,s p_x+t)-(q_x,sq_x+t)\|_2
= \sqrt{(p_x-q_x)^2 + (sp_x-sq_x)^2} 
= |p_x-q_x|\sqrt{1+s^2}.
\]

\item [(III) slice is between $p$ and $q$:]
There are two subcases, which we will call (IIIa) and (IIIb): Assuming $p$ lies above the slice (IIIa), the formula is
\begin{align*}\delta_{p,q}(s,t)&=\|(\frac{p_y-t}{s},p_y)-(q_x,sq_x+t)\|_2
=\sqrt{\left(\frac{p_y-t}{s}-q_x\right)^2 + (p_y-sq_x-t)^2}\\
&=\sqrt{\frac{1}{s^2}\left(p_y-t-sq_x\right)^2 + (p_y-sq_x-t)^2}
=  |p_y-t-sq_x|\sqrt{1+\frac{1}{s^2}}.
\end{align*}
If $p$ lies below the slice (IIIb), the formula is the same, with the roles of $p$ and $q$ exchanged.
\end{description}

The push map is easily seen to be continuous with respect to the slice $\ell$, so these formulae also extend to boundaries of the cases, i.e., when the slice contains $p$ or $q$.

\begin{lemma}
\label{lem:delta_invariance}
If $p,q\in\grade(\pres^M)\cup \grade(\pres^N)$, then in each face of $\mathcal{A}_0$, exactly one of the conditions \textup{(I), (II), (IIIa), (IIIb)} holds everywhere.  
Hence, $\delta_{p,q}$ can be expressed on the entire face by one of the analytic formulae above.
\end{lemma}
\begin{proof}
Assume for a contradiction that two points $x$, $x'$ in a face $f$
of $\mathcal{A}_0$ are of two different types among 
\textup{(I), (II), (IIIa), (IIIb)}. Let $\ell=\alpha(x)$, $\ell'=\alpha(x')$ denote the corresponding slices. Then either $p$ or $q$ (or both) switch
sides from $\ell$ to $\ell'$. Assume that $p$ lies above $\ell$
and $p$ lies below $\ell'$~-- all other cases are analogous. Then, on the
line segment from $x$ to $x'$, there is some point $x''$ such that $p$
lies on the slice $\alpha(x'')$. Hence, $x''$ lies on $L_p$ and is therefore
not in $f$. This contradicts the convexity of the faces.
\end{proof}

In view of Lemma~\ref{lem:delta_invariance}, for $p,q\in\grade(\pres^M)\cup \grade(\pres^N)$ we may define the \emph{type} of $\delta_{p,q}$ on a face of $\mathcal{A}_0$ to be the case $\textup{(I), (II), (IIIa), or (IIIb)}$ which holds on that face.

\subparagraph{Refinement of the arrangement}
Now we further subdivide the arrangement~$\mathcal{A}_0$. 
For that, consider the set of equations of the form
\begin{equation}\label{eq:delta_sol}
 \begin{gathered}
\begin{array}{rcl}
\delta_{p,q}(s,t)&=&0 \qquad \qquad\qquad\ \ \ \ \, \textup{for $p,q\in \grade(\pres^M)$ or $p,q\in \grade(\pres^N)$}, \\[0.5em]
\onehalforone_{pq}\delta_{p,q}(s,t)&=&\onehalforone_{p'q'}\delta_{p',q'}(s,t)\qquad \textup{for $p,q,p',q'\in \grade(\pres^M)\sqcup \grade(\pres^N)$},
\end{array}
\end{gathered}
\end{equation}
where
\[\onehalforone_{pq}:=\begin{cases} \frac{1}{2}&\textup{if $p,q\in \grade(\pres^M)$ or $p,q\in \grade(\pres^N)$},\\ 1 &\textup{otherwise}.\end{cases}\]

\begin{lemma}
\label{lem:solution_in_face}
The solution set of each of the above equations restricted to a face $\face$ is
 either the empty set, the entire face, the intersection of $\face$ with a line, or 
intersection of $\face$ with the union of two lines.
\end{lemma}
\begin{proof}
First we show the statement for equations of type $\delta_{p,q}(s,t)=0$.
There are $3$ cases:
\begin{description}
  \item[$\delta_{p,q}$ is of type (I):]
the equation becomes
\[|p_y-q_y|\sqrt{1+\frac{1}{s^2}}=0\]
for which either all $(s,t)\in\face$ are a solution (if $p_y=q_y$),
or no $(s,t)$ is a solution.
\item[$\delta_{p,q}$ is of type (II):]
 the same argument holds for the equation
\[|p_x-q_x|\sqrt{1+s^2}=0.\]
\item[$\delta_{p,q}$ is of type (III):]
Swapping $p$ and $q$ if necessary, we obtain the equation
\[|p_y-t-sq_x|\sqrt{1+\frac{1}{s^2}}=0\]
and the solution set is made of all $(s,t)\in\face$ for which $p_y-t-sq_x=0$, 
which is the equation of a line. 
\end{description}

For the remaining equations, we give the proof in the special case that $\onehalforone_{pq}=\onehalforone_{p'q'}$; the proof in the other cases is essentially the same.
For equations of the form $\delta_{p,q}(s,t)=\delta_{p',q'}(s,t)$,
there are six cases to check, depending on the type of the $\delta$-functions on the left and right sides of the equation:
\begin{description}
  \item[Both $\delta_{p,q}(s,t)$ and $\delta_{p',q'}(s,t)$ are of type
    (I):] the equation is
\[|p_y-q_y|\sqrt{1+\frac{1}{s^2}}=|p'_y-q'_y|\sqrt{1+\frac{1}{s^2}}\]
and the equation is satisfied if and only if $|p_y-q_y|=|p'_y-q'_y|$,
independent of $s$ and $t$. Hence, the solution set is either $\face$ or
$\emptyset$.
\item[Both $\delta_{p,q}(s,t)$ and $\delta_{p',q'}(s,t)$ are of type
    (II):]
the same argument
as in the previous case applies, so the solution set is either $\face$ or
$\emptyset$.
\item[$\delta_{p,q}$ is of type (I) and $\delta_{p',q'}$ of type
  (II):] we get the equation
\[|p_y-q_y|\sqrt{1+\frac{1}{s^2}}=|p'_x-q'_x|\sqrt{1+s^2}.\]
Since $1+\frac{1}{s^2}=\frac{1+s^2}{s^2}$, this simplifies to
\[|p_y-q_y| = s |p'_x-q'_x|\]
and the solution set is either all of $\face$ (if both absolute values vanish),
the empty set (if only $p'_x-q'_x=0$), or the intersection of $\face$
with the vertical line $s=\frac{|p_y-q_y|}{|p'_x-q'_x|}$ (otherwise).
\item[Both $\delta_{p,q}$ and $\delta_{p',q'}$ are of type (III):]
 Swapping $p,q$ or $p',q'$ if necessary, we get
\[|p_y-t-sq_x|\sqrt{1+\frac{1}{s^2}}=|p'_y-t-sq'_x|\sqrt{1+\frac{1}{s^2}}\]
hence, $(s,t)$ is a solution if and only if $p_y-t-sq_x=p'_y-t-sq'_x$
or $p_y-t-sq_x=-(p'_y-t-sq'_x)$. The first equations yields again either
$\face$, $\emptyset$, or  a vertical line as solution set, the second equation
always defines a line. Exchanging the roles of $p$ and $q$, or the roles of $p'$ and $q'$, or both, does not change the conclusion.
\item[$\delta_{p,q}$ is of type (I) and $\delta_{p',q'}$ is of type (III):]
Swapping $p'$ and $q'$ if necessary, 
the formula is
\[|p_y-q_y|\sqrt{1+\frac{1}{s^2}} = |p_y-t-sq_x|\sqrt{1+\frac{1}{s^2}}.\]
$(s,t)\in\face$ is a solution if $p_y-t-sq_x=p_y-q_y$ or $p_y-t-sq_x=q_y-p_y$,
which is a line equation in both cases. 
\item[$\delta_{p,q}$ is of type (II) and $\delta_{p',q'}$ is of type (III):]
Swapping $p'$ and $q'$ if necessary, we get
\[|p_x-q_x|\sqrt{1+s^2} = |p'_y-t-sq'_x|\sqrt{1+\frac{1}{s^2}}\]
which simplifies to
\[s |p_x-q_x| = |p'_y-t-sq'_x|\]
Here $(s,t)\in\face$ is a solution if and only if $s(p_x-q_x)=p'_y-t-sq'_x$
or $s(p_x-q_x)=-(p'_y-t-sq'_x)$. Again, we obtain a line in both cases.\qedhere
\end{description}
\end{proof}

\begin{definition}\label{def:A}
Let $\mathcal{A}$ denote the line arrangement in $\domain$ formed by the lines in $\mathcal{A}_0$, all lines from the case analysis above, and the vertical line $s=1$.
\end{definition}

\begin{lemma}\label{lem:sizeofA}
The arrangement $\mathcal{A}$ consists of $O(n^4)$ lines.
\end{lemma}

\begin{proof}
The case analysis in the proof of Lemma~\ref{lem:solution_in_face}  was performed relative to a choice of face $f$ in $\mathcal A_0$.  However, for a fixed choice of an equation in the set of equations \eqref{eq:delta_sol}, the lines which arise in the case analysis depend only on the types of $\delta_{p,q}$ and $\delta_{p',q'}$ on $f$.  There are at most $4\times 4=16$ possible ways of jointly choosing those types, and for a given choice, at most two lines are added to the arrangement.  Hence, each of the $O(n^4)$ equations in the set of equations \eqref{eq:delta_sol} contributes at most a constant number of lines to $\mathcal A$.  The result now follows easily.
\end{proof}

Note that the arrangement $\mathcal A$ depends on the choice of presentations for $M$ and $N$. The next statement says that within a face of the arrangement,
the bottleneck distance is realized by the difference of the pushes
of two fixed grades of the presentations.

\begin{theorem}
\label{thm:delta_theorem}
For any face $f$  of $\mathcal{A}$, there is some choice of $p,q\in \grade(\pres^M)\cup \grade(\pres^N)$ such that $d_B(M^\ell,N^\ell)=c_{pq}\delta_{p,q}(\ell)$ for all $\ell\in \alpha(f)$.
\end{theorem}

\iffullversion
\noindent The formal proof of this result is deferred to Appendix~\ref{sec:delta_theorem}. It can be summarized as follows:
\else
\noindent See the full version~\cite[App. A]{klo-full} for a complete proof.
It can be summarized as follows:
\fi
By Remark \ref{Rem:Push_Pres} and Lemma \ref{Lem:Pres_Grades_And_Persistence_Diagram}, for each $\ell\in\face$  there is a collection $T_M^\ell$ of pairs \[(b,d) \in \grade(P^M)\times (\grade(P^M)\cup\{(\infty,\infty)\})\]  such that $D(M^\ell)$ is obtained by pushing the elements of $T_M^\ell$ onto $\ell$.    Call such $T_M^\ell$ a \emph{template} for $D(M^\ell)$.  For $\ell$ and $\ell'$ such that the grades of $P^M$ push onto~$\ell$ and~$\ell'$ in the same order, the templates for $D(M^\ell)$ and $D(M^{\ell'})$ are the same.  But in fact, the grades of $P^M$ push onto all lines $\ell\in \face$ in the same order, because whenever the order changes, we need to cross one of the lines of the arrangement $\mathcal{A}$.  Hence, there exists $T_M^\face$ that is a template for $D(M^\ell)$, for all $\ell\in\face$.  Similarly for $N$.  

By Lemma \ref{lem:bottelneck_structure}, for any fixed $\ell'\in \face$ we have that $d_B(M^{\ell'},N^{\ell'})=c_{pq}\delta_{p,q}(\ell')$, for $p$ and $q$ each some coordinate of a pair in $T_M^\face\cup T_N^\face$. The order of the values of the functions
    \[\{c_{rs}\delta_{r,s}\mid r,s\in \grade(\pres^M)\cup \grade(\pres^N)\}\] remains the same
    across~$\face$, because any change in the order will result
    again in crossing one of the lines of the arrangement $\mathcal{A}$.  
Thus, in fact $d_B(M^{\ell},N^{\ell})=c_{pq}\delta_{p,q}({\ell})$ for all $\ell\in \face$.

\section{Maximization}
\label{sec:maximization}
We define a \emph{region} of $\mathcal{A}$ as the closure of a face of $\mathcal{A}$
within $\domain$.
We can compute the matching distance by determining $\sup F(s,t)$ separately for
each region in $\mathcal{A}$. 
We will show now that in each region, $\sup F(s,t)$ is either realized
at a boundary vertex, or as the limit of an unbounded boundary edge, which can
be computed easily.

We fix the following notation:
A region $\region\subseteq\domain$ is an \emph{inner region} if it is bounded
as a set in $\R^2$ and it has a positive distance to the vertical line $s=0$ in $\R^2$
(in other words, $\region$ does not reach the boundary of $\domain$). 
An inner region is a convex polygon.
Regions that are not inner region are called \emph{outer regions}.
Outer regions have exactly two \emph{outer segments} in their boundary,
which are infinite or converge to a point on the vertical line $s=0$
(except for the empty arrangement, which only consists of one outer region).
See Figure~\ref{fig:segments} for an illustration.

\begin{figure}
\centering
\includegraphics[width=5cm]{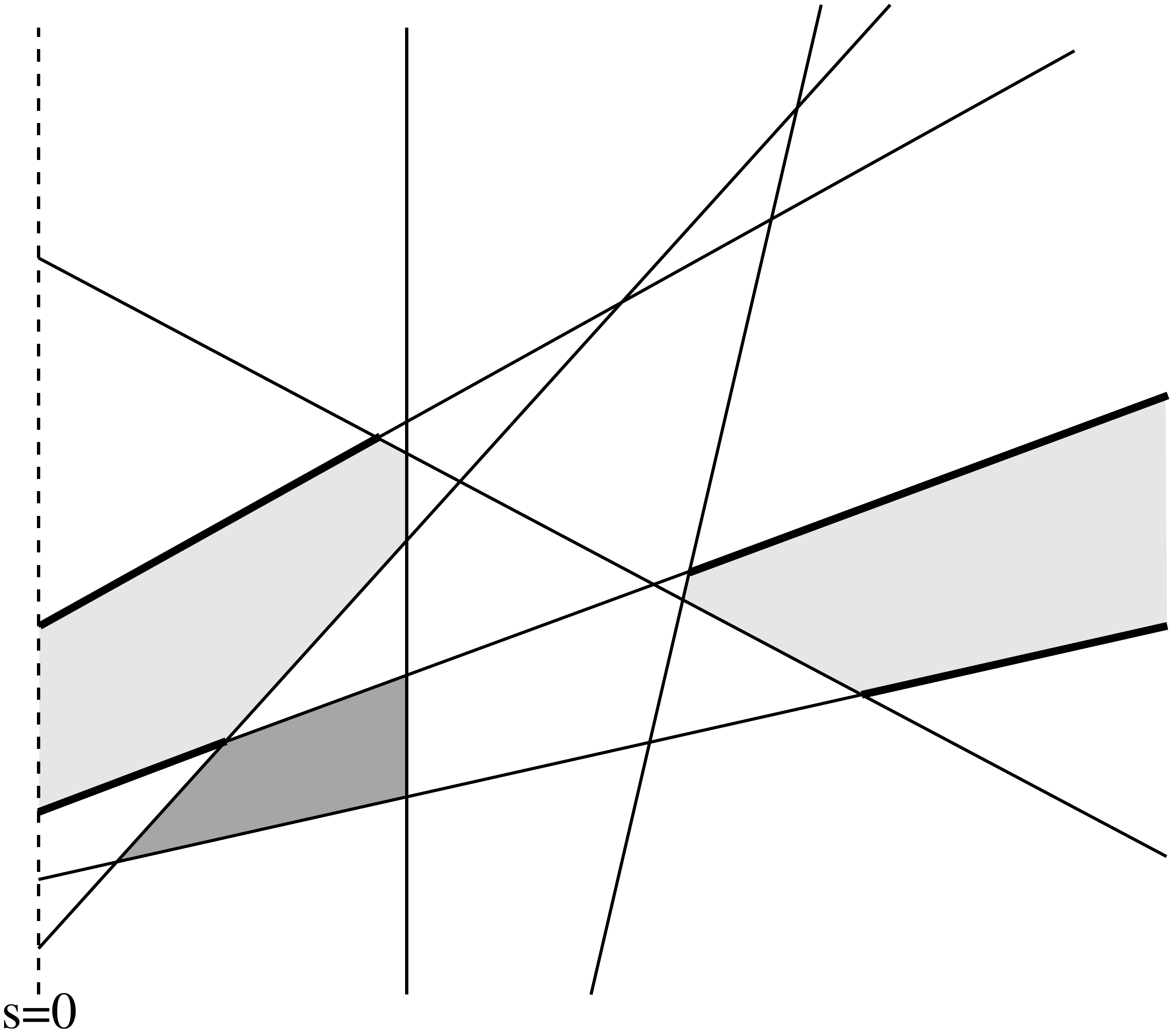}
\caption{The lightly shaded regions show outer regions.  The outer segments of these are drawn more thickly.  The darkly shaded region is an example of an inner region.}
\label{fig:segments}
\end{figure}

For a fixed region $\region$ of $\mathcal{A}$, \cref{thm:delta_theorem}
ensures that there is a pair of grades $(p,q)$ 
whose $\delta$-function realizes $d_B$ within the interior of $\region$ (a face
of $\mathcal{A}$). By continuity, this implies 
that $\delta_{p,q}$ also realizes $d_B$ on the entire region.

\begin{lemma}\label{lem:maximum_region}
The supremum of $F$ within $\region$ is attained either at a vertex on
the boundary of~$\region$, or as the limit of $F$ along an outer segment.
In the latter case, the limit can
be expressed in simple terms based on the equation of the 
line segment and one of the functions~$\delta_{p,q}$.
\end{lemma}
\begin{proof}
We distinguish $6$ cases based on the type of the $\delta$-function
and on whether $s\leq 1$ or $s\geq 1$ (note that each cell belongs to
one of these cases, because the line $s=1$ is in $\mathcal{A}$).

\begin{description}
\item [$\delta_{p,q}$ of type (I), $s\leq 1$] In that case, 
\[F(\ell)=w(\ell)\delta_{p,q}(\ell)=\frac{1}{\sqrt{1+\frac{1}{s^2}}}|p_y-q_y|\sqrt{1+\frac{1}{s^2}}=|p_y-q_y|,\]
a constant function. Clearly, the supremum is attained everywhere, in particular at the boundary vertices of $\region$.

\item [$\delta_{p,q}$ of type (I), $s\geq 1$] We get
\[F(\ell)=\frac{1}{\sqrt{1+s^2}}|p_y-q_y|\sqrt{1+\frac{1}{s^2}}=\frac{1}{s}|p_y-q_y|\]
Clearly, this function becomes larger when $s$ gets smaller. Moreover,
because $s\geq 1$ within the cell, there is a leftmost vertex on the
boundary, which minimizes $s$ and therefore attains the supremum
within the cell.

\item [$\delta_{p,q}$ of type (II), $s\leq 1$] We obtain
\[F(\ell)=\frac{1}{\sqrt{1+\frac{1}{s^2}}}|p_x-q_x|\sqrt{1+s^2}=s|p_x-q_x|.\]
Similarly to the previous case, there exists a rightmost boundary vertex in the cell (because $s\leq 1$), which realizes the supremum.

\item [$\delta_{p,q}$ of type (II), $s\geq 1$] The function simplifies to
\[F(\ell)=\frac{1}{\sqrt{1+s^2}}|p_x-q_x|\sqrt{1+s^2}=|p_x-q_x|,\]
a constant function, which attains its supremum at any boundary vertex.

\item [$\delta_{p,q}$ of type (III), $s\leq 1$] Assuming that $p$ lies above the slice, we get
\[F(\ell)=\frac{1}{\sqrt{1+\frac{1}{s^2}}}|p_y-t-sq_x|\sqrt{1+\frac{1}{s^2}}
=|p_y-t-sq_x|.\]
If $\region$ is an inner region, we are maximizing the above function
over a closed convex polygon, and the maximum
is achieved at a boundary vertex, because $|p_y-t-sq_x|$
is the maximum of two linear functions in $s$ and $t$.

It remains to analyze the case that $\region$ is an outer region.
We argue first that $\region$ is bounded in $t$-direction from above and below:
Since $\delta_{p,q}$ is of type (III), with $p$ lying above $\ell$, 
 $(s,t)$ must be below the 
(non-vertical) line $t=-sp_x+p_y$ in the dual space.
Likewise, since $q$ is below $\ell$,
$(s,t)$ must be above $t=-sq_x+q_y$.
Moreover, we have $0<s\leq 1$.
If the above lines intersect at a point $r$ with $s$-value in $(0,1)$,
$\region$ is contained in the triangle spanned by the two lines
and the vertical line $s=0$. Otherwise, $\region$ is contained
in the trapezoid induced by these two lines and the vertical lines 
$s=0$ and $s=1$.

It follows that the two outer segments of $\region$ converge to the
vertical line $s=0$. Let $(0,t_1)$ denote the limit of the lower outer
segment and $(0,t_2)$ the limit of the upper outer segment.  Clearly
$t_1\leq t_2$.  Let $\bar{\region}$ denote the union of $\region$
with the vertical line segment from $(0,t_1)$ to $(0,t_2)$; note that
$\bar{\region}$ is the closure of $\region$ considered as a subset of
$\R^2$.  Observe that $|p_y-t-sq_x|$ is continuous over $\R^2$;
therefore $F$ can be continuously extended to $\bar{\region}$.  It
follows that the supremum of $F$ over $\bar{\region}$ is attained at a
boundary vertex, since $\bar{\region}$ is a convex closed polygon.
There are two cases: either the maximum is attained at a vertex of
$\mathcal{A}$, or at $(0,t_1)$ or $(0,t_2)$. As we can readily see,
the function values at the latter two points are $|p_y-t_1|$ and
$|p_y-t_2|$, respectively. The case where $p$ is below the slice and
$q$ is above is analyzed in the same way, \mbox{with the roles of $p$ and $q$
swapped}.

\item [$\delta_{p,q}$ of type (III), $s\geq 1$] Assuming that $p$ lies above the slice, we get
\[F(\ell)=\frac{1}{\sqrt{1+s^2}}|p_y-t-sq_x|\sqrt{1+\frac{1}{s^2}}=
|\frac{p_y}{s}-\frac{t}{s}-q_x|\]

We first consider the case where $\region$ is an inner region, and we
show that the function is maximized at a boundary vertex of $\region$.
The function $|\frac{p_y}{s}-\frac{t}{s}-q_x|$ has no local maximum
over $\R^2$ since it is the absolute value of a linear function in $t$
for any fixed $s$.  Hence, the supremum over $\region$ must be
attained on the boundary. We have to exclude the case that the maximum
lies in the interior of an edge. For vertical edges this is obvious,
because for a constant $s$, the function simplifies to the absolute
value of a linear function in $t$ which must be maximized at a
boundary vertex.  For a non-vertical line of the form $t=as+b$,
plugging in this equation for $t$ yields a function of the form
\[|\frac{p_y}{s} -\frac{as+b}{s}-q_x|=|\frac{1}{s}(p_y-b) - a -q_x|.\] 
This is the absolute value of a monotone function in $s$ and hence has no local maximum.
Again, this implies that it is maximized at a boundary vertex.

\medskip

Consider now the case where $\region$ is an outer region.
As in the previous case, $\region$ is upper and lower bounded by two non-vertical
lines, because we assume type (III).
Hence, the two outer segments of $\region$ cannot be vertical;
the lower outer segment has a slope $r_1$ and the upper outer segment has a slope $r_2$
with $r_1<r_2$.
We argue next that the supremum of $\frac{p_y}{s}-\frac{t}{s}-q_x$ is either attained
at a boundary vertex, or equal to $|r_1+q_x|$, or equal to $|r_2+q_x|$.
Let $(s_i,t_i)$ denote a sequence of points in $\region$ 
such that $F(s_i,t_i)$ converges to the supremum.
If $(s_i,t_i)$ converges to a point in $\region$ (or has at least a convergent
subsequence), it follows (similarly to the case of an inner region) that 
the limit point is a boundary vertex. Otherwise, we can assume (by passing to
a subsequence) that the sequence $s_i$ is unbounded. Moreover, 
the sequence $\frac{t_i}{s_i}$ is bounded by $[r_1,r_2]$ and therefore has
a convergent subsequence with limit $r'$. Passing to this subsequence, we obtain that
\[\lim_{i\to\infty} F(s_i,t_i)=\lim_{i\to\infty} |\frac{p_y}{s_i}-\frac{t_i}{s_i}-q_x|
=|-r'-q_x|=|r'+q_x|\]
Hence, the supremum must be of the form $|r'+q_x|$ for some $r'\in[r_1,r_2]$.
On the other hand, this expression is clearly maximized for either
$|r_1+q_x|$ or $|r_2+q_x|$, and there exist sequences attaining these values,
for instance when choosing $(s_i,t_i)$ on either of the outer segments.
The case where $p$ lies below the slice and $q$ lies above is treated similarly, with the roles of $p$ and $q$ exchanged.\qedhere
\end{description}
\end{proof}

\subparagraph{The algorithm} We now give the algorithm to compute the matching distance:

\begin{itemize}
\item Compute the arrangement induced by $\mathcal{A}$ from Definition~\ref{def:A}.
\item For each vertex $(s,t)$ in the arrangement, compute $F(s,t)$. Let $m$ be the maximum among all the values.
\item For each outer region $\region$,
pick a point $(s,t)$ in the interior. 
Compute the bottleneck distance and identify a pair $(p,q)$ of grades
that realizes the bottleneck.
Determine whether $p$ and $q$ are
above or below the slice $(s,t)$.  If the region is of type (IIIa) with
respect to $p$ and $q$, then do the following:
\begin{itemize}
\item If $\region$ is to the left of $s=1$, compute the intersections $(0,t_1)$,
$(0,t_2)$ of the 
outer segments of $\region$ with the vertical line $s=0$.
Set $m\gets\max\{m,\ |p_y-t_1|,\ |p_y-t_2|\}$.
\item If $\region$ is to the right of $s=1$, let $r_1$, $r_2$ denote the slopes of
  the outer segments of $\region$. Set $m\gets\max\{m,\ |r_1+q_x|,\ |r_2+q_x|\}$.
  \end{itemize}
If the region is of type (IIIb), proceed analogously, with the roles
of $p$ and $q$ exchanged.
\item Return $m$.
\end{itemize}

By ``computing the arrangement'', we mean to store the planar subdivision induced by the lines of the arrangement (e.g.\cite[Ch.2]{dutch}).
In fact, it is possible to implement the algorithm without explicitly constructing the line arrangement $\mathcal A$ 
or even storing its entire set of vertices.  This reduces the space
complexity of the algorithm.
\iffullversion
See Appendix~\ref{app:space_efficient} for the details.
\else
See~\cite[App. B]{klo-full} for details.
\fi
\begin{theorem}
The above algorithm computes the matching distance in polynomial time.
\end{theorem}
\begin{proof}
Correctness follows from Lemma \ref{lem:maximum_region}: as we check all
vertices of the arrangement, we cover the supremum of all inner regions.
The outer regions are handled separately
in the last steps of the algorithm.

Running time: recall from Lemma~\ref{lem:sizeofA} that we have $O(n^4)$ lines in the arrangement~$\mathcal A$. 
Hence, the arrangement has $O(n^8)$
vertices, $O(n^4)$ outer regions, and can be computed in $O(n^8\log n)$ time
using an extension of the Bentley-Ottman sweep-line algorithm~\cite{bo-arcgi-79}. 
For each vertex and each
outer region, we have to compute two persistence diagrams, 
which can be done in $O(n^3)$ time, and a bottleneck distance whose
complexity can be neglected. The remaining computations are negligible.
Hence, we arrive at a $O(n^{11})$ algorithm.
\end{proof}

We remark that the algorithm can be realized entirely with
rational arithmetic if all grades are rational numbers. 
Indeed, all lines in the arrangement
have rational coefficients, and so do their intersection points.
An intersection points corresponds to a slice along which we are required
to compute the bottleneck distance. 
Recall from \cref{sec:matching_distance} that the definition of the slice $\ell$
module introduces a grade of $\lambda^{-1}(\push(p,\ell))$ where $\lambda^{-1}(p_x,p_y)=\sqrt{1+s^2} p_x$.
Hence, these grades are not rational numbers.
However, the bottleneck distance is multiplied with the weight $w(\ell)$ of the slice afterwards,
and instead of doing so, one can as well scale all grades with $w(\ell)$ in advance.
A simple calculation shows that this indeed turns the grades into rational values.

A simple analysis also reveals that if the input coordinates are rational and of bitsize $\leq b$, 
all intermediate computations in the algorithm can be performed with a bitsize of $\leq cb$,
with $c$ a (small) constant. Hence, the algorithm is strongly polynomial.

\section{Discussion}
\label{sec:discussion}
We have presented the first polynomial time algorithm to exactly compute the matching
distance for 2-parameter persistence modules. 
It is natural to ask about practicality of our approach.
The large exponent of $n^{11}$ seems discouraging at first, but we mention
first that the worst-case running time of $O(n^3)$ for persistent homology
is usually not appearing for real instances; indeed an almost linear behavior
can be expected.
Still, the large number of $O(n^4)$ lines in the arrangement 
constitutes a computational barrier in practice. 
There are several possibilities, however, to reduce this effect:
\begin{itemize}
\item Instead of computing the arrangement $\mathcal{A}$ globally,
we could compute the intermediate arrangement $\mathcal{A}_0$
and refine each face of it separately, using only those lines
that affect the $\delta$-functions within this face.
\item As a follow-up to the previous point, it might be possible
to compute a smaller arrangement per face adaptively. The idea is
to start at some interior point in a face of $\mathcal{A}_0$,
identifying a pair $(p,q)$ that realizes the bottleneck
distance and then to determine the boundary of the region where
$(p,q)$ realizes the bottleneck distance.  

\item As a preprocessing step, we can minimize the input presentations, yielding presentations for the same modules with the smallest possible number of generators and relations 
(hence minimizing~$n$) \cite{lesnick2018computing}.  
\end{itemize}
We pose the question of whether an implementation realizing the above ideas
is competitive to an approximative, sampling-based approach 
for computing the matching distance.

Our algorithm needs to treat the outer edges of the arrangement $\mathcal{A}$
separately since our analysis does not rule out the possibility
that the supremum is realized at the boundary of $\domain$. On the other hand,
we are not aware of an example of two finite presentations
whose matching distance is not realized by a particular slice in $\domain$.
A proof that the supremum in the definition of the matching distance
is in fact a maximum would greatly simplify our algorithm, since it would
boil down to computing the intersection points of all lines and 
searching for the maximal $F$-value among them. 

Finally, we have restricted attention to the case of two-parameter persistence modules.
It is natural to conjecture that our algorithm extends to more parameters by constructing a hyperplane arrangement.  It would be worthwhile to check this conjecture in future work.

\bibliography{bib}

\begin{thebibliography}{10}

\bibitem{bo-arcgi-79}
J.~Bentley and T.~Ottmann.
\newblock {Algorithms for Reporting and Counting Geometric Intersections}.
\newblock {\em IEEE Transactions on Computers}, 28:643--647, 1979.

\bibitem{dutch}
M.~de Berg, O.~Cheong, M.~van Kreveld, and M.~Overmars.
\newblock {\em Computational Geometry: Algorithms and Applications}.
\newblock Springer-Verlag TELOS, Santa Clara, CA, USA, 3rd ed. edition, 2008.

\bibitem{BiasottiCerriFrosini2011}
S.~Biasotti, A.~Cerri, P.~Frosini, and D.~Giorgi.
\newblock A new algorithm for computing the 2-dimensional matching distance
  between size functions.
\newblock {\em Pattern Recognition Letters}, 32(14):1735--1746, 2011.

\bibitem{Bjerkevik}
H.~Bjerkevik, M.~Botnan, and M.~Kerber.
\newblock Computing the interleaving distance is {NP}-hard.
\newblock {\em arXiv}, abs/1811.09165, 2018.
\newblock URL: \url{http://arxiv.org/abs/1811.09165}.

\bibitem{Carlsson2009}
G.~Carlsson and A.~Zomorodian.
\newblock The theory of multidimensional persistence.
\newblock {\em Discrete \& Computational Geometry}, 42(1):71--93, 2009.

\bibitem{cerri2013betti}
A.~Cerri, B.~Di~Fabio, M.~Ferri, P.~Frosini, and C.~Landi.
\newblock Betti numbers in multidimensional persistent homology are stable
  functions.
\newblock {\em Mathematical Methods in the Applied Sciences},
  36(12):1543--1557, 2013.

\bibitem{Corbet2018}
R.~Corbet and M.~Kerber.
\newblock The representation theorem of persistence revisited and generalized.
\newblock {\em Journal of Applied and Computational Topology}, 2(1):1--31,
  2018.

\bibitem{crawley2015decomposition}
W.~Crawley-Boevey.
\newblock Decomposition of pointwise finite-dimensional persistence modules.
\newblock {\em Journal of Algebra and its Applications}, 14(05):1550066, 2015.

\bibitem{RIVET}
The~RIVET Developers.
\newblock {RIVET}: Software for visualization and analysis of 2-parameter
  persistent homology.
\newblock \url{http://repo.rivet.online/}, 2014-2018.

\bibitem{dx-computing}
T.~Dey and C.~Xin.
\newblock Computing bottleneck distance for 2-{D} interval decomposable
  modules.
\newblock In {\em International Symposium on Computational Geometry (SoCG)},
  pages 32:1--32:15, 2018.

\bibitem{EdelsbrunnerHarer2010}
H.~Edelsbrunner and J.~Harer.
\newblock {\em Computational Topology: An Introduction}.
\newblock American Mathematical Society, Providence, RI, USA, 2010.

\bibitem{eik-geometry}
A.~Efrat, A.~Itai, and M.~Katz.
\newblock Geometry helps in bottleneck matching and related problems.
\newblock {\em Algorithmica}, 31(1):1--28, 2001.

\bibitem{Keller2018}
B.~Keller, M.~Lesnick, and T.~L. Willke.
\newblock {Persistent Homology for Virtual Screening}.
\newblock {\em ChemRxiv preprint}, 10 2018.

\bibitem{kmn-geometry}
M.~Kerber, D.~Morozov, and A.~Nigmetov.
\newblock Geometry helps to compare persistence diagrams.
\newblock {\em Journal of Experimental Algorithms}, 22:1.4:1--1.4:20, September
  2017.

\bibitem{Landi2018}
C.~Landi.
\newblock The rank invariant stability via interleavings.
\newblock In {\em Research in Computational Topology}, pages 1--10. Springer
  International Publishing, 2018.

\bibitem{Lesnick2015}
M.~Lesnick.
\newblock The theory of the interleaving distance on multidimensional
  persistence modules.
\newblock {\em Foundations of Computational Mathematics}, 15(3):613--650, 2015.

\bibitem{Lesnick2015b}
M.~Lesnick and M.~Wright.
\newblock Interactive visualization of {2-D} persistence modules.
\newblock {\em arXiv:1512.00180}, 2015.

\bibitem{lesnick2018computing}
M.~Lesnick and M.~Wright.
\newblock Computing minimal presentations and {B}etti numbers of 2-parameter
  persistent homology.
\newblock {\em arXiv:1902.05708}, 2019.

\bibitem{mms-zigzag}
N.~Milosavljevic, D.~Morozov, and P.~Skraba.
\newblock Zigzag persistent homology in matrix multiplication time.
\newblock In {\em {ACM} Symposium on Computational Geometry (SoCG)}, pages
  216--225, 2011.

\bibitem{Oudot2015}
S.~Oudot.
\newblock {\em Persistence theory: From Quiver Representation to Data
  Analysis}, volume 209 of {\em Mathematical Surveys and Monographs}.
\newblock American Mathematical Society, 2015.

\bibitem{ZomorodianCarlsson2005}
A.~Zomorodian and G.~Carlsson.
\newblock Computing persistent homology.
\newblock {\em Discrete and Computational Geometry}, 33(2):249--274, 2005.

\end{thebibliography}

\iffullversion


\setcounter{section}{0}
\renewcommand\thesection{\Alph{section}}
\section{Appendix: proof of Theorem~\ref{thm:delta_theorem}}
\label{sec:delta_theorem}

Let $G^M=\grade(\pres^M)$ and let
$G^N=\grade(\pres^N)$.  As an intermediate
result, we prove that the combinatorial structure of each persistence 
diagram stays the same across the face~$\face$:

\begin{lemma}\label{Lem:diagram_template}
For each face $f$ of $\mathcal A$, there exist multisets
\begin{align*}
\mathcal T_M^f&\subset G^M\times (G^M\cup \{(\infty,\infty)\}),\\
\mathcal T_N^f&\subset G^N\times (G^N\cup \{(\infty,\infty)\})
\end{align*}
such that for any $\ell \in \alpha(f)$, 
\begin{align*}
D(M^\ell)&= \left\{(\lambda^{-1} \circ \push(a,\ell),\lambda^{-1} \circ \push(b,\ell))\mid (a,b) \in \mathcal T_M^f\right\},\\
D(N^\ell)&= \left\{(\lambda^{-1} \circ \push(a,\ell),\lambda^{-1} \circ \push(b,\ell))\mid (a,b) \in \mathcal T_N^f\right\},
\end{align*}
where by convention $\push((\infty,\infty),\ell) := (\infty,\infty)$ and $\lambda^{-1}((\infty,\infty)) := \infty$. 
We call $\mathcal T_M^f$ and $\mathcal T_N^f$ \emph{diagram templates}.
\end{lemma}

\begin{proof}
This is a variant of \cite[Theorem 4.1]{Lesnick2015b}.  We give a succinct, algorithmically flavored proof here.  We prove the result for $M$; the proof for $N$ is the same.   

 First, note that the restriction of the partial order on $\R^2$ to any $\ell\in \alpha(f)$ is a total order.  This total order pulls back under the map $\push(-,\ell):\grade(P^M)\to \ell$ to a totally ordered partition of $\grade(P^M)$, (i.e., elements of the partition are level sets). It can be checked that, because $\mathcal A$ refines $\mathcal A_0$ and also contains all lines of the form $\delta_{p.q}=0$ for $p,q\in \grade(P^M)$ with $p$ and $q$ incomparable in the partial order on $\R^2$, this totally ordered partition is the same for all $\ell\in \alpha(f)$; see \cite[Corollary 3.4]{Lesnick2015b}.  
Thus, we obtain a unique totally ordered partition of $\grade(P^M)$ associated to all of $f$.  Let us refine this to a fixed total order on $\grade(P^M)$.  
 
 Permuting the row or columns of a presentation for $M$ yields another presentation for $M$, so we may assume without loss of generality that the order of rows and columns for $P^M$ is consistent with our total order on $\grade(P^M)$.  Applying the matrix reduction underlying the standard persistence algorithm \cite{ZomorodianCarlsson2005} to the matrix underlying $P^M$ yields a matching $\sigma$ of row and column indices, where a non-zero column in the reduced matrix is matched to the row of its pivot.  We may define 
\begin{align*}
\mathcal T_M^f:&=\left\{\grade(\row_i),\grade(\col_j) \mid (i,j)\in \sigma,\ \grade(\row_i),\grade(\col_j)\textup{ not together in the partition} \right\}\\
 &\cup \left\{(\grade(\row_i),(\infty,\infty))\mid i\textup{ is unmatched in }\sigma\right\}.
\end{align*}
By Remark \ref{Rem:Push_Pres} and the correctness of the standard algorithm for computing persistent homology \cite{ZomorodianCarlsson2005}, it follows that 
\[D(M^\ell)= \left\{(\lambda^{-1} \circ \push(a,\ell),\lambda^{-1} \circ \push(b,\ell))\mid (a,b) \in \mathcal T_M^f\right\},\]
as desired.
\end{proof}

Suppose we are given two persistence diagrams $D$ and $D'$.  For each $(x,y)\in D\cup D'$, let $w_0(x,y):=\frac{y-x}{2}$, and for each $(x,y),(x',y')\in D\times D'$, let $w_1(x,x'):=\max(|x-x'|)$, and $w_2(y,y'):=\max(|y-y'|)$.  We call these (possibly infinite) numbers the \emph{weights} of the pair $(D,D')$.  Denoting the set of pairs indexing these weights as $I(D,D')$, the weights define a function $w^{D,D'}:I^{D,D}\to [0,\infty)$.  Note that $d_B(D,D')= w^{D,D'}(x,y)$ for some $(x,y)\in I^{D,D'}$.   

Clearly, a pair of bijections of persistence diagrams $\zeta_1:D\to E$ and $\zeta_2:D'\to E'$ induces a bijection \[\zeta:I(D,D')\to I(E,E').\] 
The following is an easy consequence of the definition of $d_B$. 
\begin{lemma}\label{Lemma:dB_and_Ordering}
If $d_B(D,D')=w^{D,D'}(x,y)$ and $\zeta_1$, $\zeta_2$ preserve the order of the weights, in the sense that \[w^{D,D'}(x,y)\leq w^{D,D'}(x',y')\quad \textup{if and only if} \quad w^{E,E'}(\zeta(x,y))\leq w^{E,E'}(\zeta(x',y')),\] 
then $d_B(E,E')=w^{E,E'}(\zeta(x,y))$.
\end{lemma}

\begin{proof}[Proof of Theorem~\ref{thm:delta_theorem}]
Let $\mathcal T_M^f$ and $\mathcal T_N^f$ be diagram templates for the cell $f$.  For any $\ell\in \alpha(f)$, Lemma~\ref{Lem:diagram_template} gives us distinguished bijections 
\[\gamma:\mathcal T_M^f\to D(M^\ell), \qquad \gamma':\mathcal T_N^f\to D(N^\ell).\]
It is easily checked that for $(x,y)\in D(M^\ell)$, and $(p,q)=\gamma^{-1}(x,y)$, we have 
\begin{equation}\label{eq:weights1}
w_0(x,y)=\frac{1}{2}\delta_{p,q}(\ell)=c_{pq}\delta_{p,q}(\ell).
\end{equation}
 Similarly,  for 
\[(x,y)\in D(M^\ell),\ (x',y')\in D(N^\ell),\ (p,p')=\gamma^{-1}(x,y),\textup{ and }(q,q')=\gamma'^{-1}(x',y'),\] we have 
\begin{align}
\label{eq:weights2}
w_1(x,x')&=\delta_{p,q}(\ell)=c_{pq}\delta_{p,q}(\ell),\\
w_2(y,y')&=\delta_{p',q'}(\ell)=c_{p'q'}\delta_{p',q'}(\ell).\label{eq:weights3}
\end{align}

Note that as we move the slice~$\ell$ inside~$f$, the order of the values taken by the functions \[\{c_{pq}\delta_{p,q}\mid p,q\in P^M\cup P^N\}\] cannot change.  Indeed, the functions $\ell \to \delta_{p,q}(\ell)$ are continuous, and by Lemma \ref{lem:solution_in_face} and the definition of the arrangement $\mathcal A$, the intersection of $\face$ with the solution set of each equation $c_{pq}\delta_{p,q}=c_{p'q'}\delta_{p',q'}$ is either $\face$ or $\emptyset$.

The diagram templates provide bijections $D(M^\ell)\to D(M^{\ell'})$ and $D(N^\ell)\to D(N^{\ell'})$ for all $\ell,\ell'\in\alpha(f)$.  
 It now follows from equations (\ref{eq:weights1}-\ref{eq:weights3}) that for each $\ell,\ell'$, these preserve the order on weights.
Thus, by Lemma \ref{Lemma:dB_and_Ordering}, if \[d_B(D(M^{\ell}),D(N^{\ell}))=w^{D(M^{\ell}),D(N^{\ell})}(x,y)\] for some $\ell \in \alpha(f)$, then \[d_B(D(M^{\ell'}),D(N^{\ell'}))=w^{D(M^{\ell'}),D(N^{\ell'})}(\zeta^{\ell,{\ell'}}(x,y))\] for all $\ell'\in \alpha(f)$, where \[\zeta^{\ell,{\ell'}}: I^{D(M^{\ell}),D(N^{\ell})}\to I^{D(M^{\ell'}),D(N^{\ell'})}\] is the bijection induced by the diagram templates.  The result now follows from equations (\ref{eq:weights1}-\ref{eq:weights3}).
\end{proof}

\section{A more space-efficient algorithm}
\label{app:space_efficient}
The algorithm described in Section~\ref{sec:maximization} consists of two major steps:
we first find the maximal value of $F$ over all intersection points of lines in
the arrangement $\mathcal{A}$, and then check the convergence along outer segments
for certain outer regions of $\mathcal{A}$. To decide whether an outer region
needs to be considered, we require one interior point of that region, to decide
the type of the region.
The arrangement $\mathcal{A}$ (stored as a planar subdivision as in~\cite[Ch.2]{dutch})
contains enough information
to access all necessary data conveniently. However, its space complexity
is $O(n^8)$ because the arrangement is induced by $O(n^4)$ lines.  
We show next how to implement the algorithm without storing $\mathcal{A}$
in memory, yielding a space complexity of $O(n^4)$:

Reporting the intersection points of a set of lines in the plane is
one of the oldest problems in computational geometry. The
Bentley-Ottmann sweep-line algorithm~\cite{bo-arcgi-79} reports all
intersection points in time proportional to the number of such
points.  Note that there are up to $O(n^8)$ intersections, so storing
all such points is space consuming.  However, there is no need to do
so~-- whenever the sweep-line algorithm reports a new point, we
compute $F$ at this point and compare the value with the maximal
$F$-value seen before, updating the maximum if the newly encountered
value is larger.  There is no need to store the intersection beyond
this moment.  In this way, we obtain the maximal $F$-value among all
intersection points using only $O(n^4)$ space.

For handling outer regions, note that we can extend the sweep-line algorithm such that it also returns
the leftmost intersection point with positive $s$-coordinate, and the rightmost intersection point, among all pairs of 
lines in $\mathcal A$. Let $0<s_{min}\leq s_{max}$ denote the $s$-coordinates of these intersection points.
Now, consider the vertical line $s=\frac{s_{min}}{2}$, which is partitioned into open intervals
\begin{equation}\label{eq:interval_sequence}
(-\infty,b_1),(b_1,b_2),\ldots,(b_{m-1},b_m),(b_m,\infty)
\end{equation}
where $b_1,\ldots,b_m$ are the intersection points of the non-vertical lines of $\mathcal{A}$ with the vertical line.  Note that $m=O(n^4)$.
The intervals in the sequence can be computed in $O(n^4 \log n)$ time, just by sorting the $b_i$'s.
The intervals in the sequence (\ref{eq:interval_sequence}) are in one-to-one correspondence with the outer regions of $\mathcal{A}$
whose boundary in $\R^2$ has nonempty intersection with the line $s=0$. More precisely, each interval $(b_i,b_{i+1})$ is contained in a unique outer region $R$; for $i\in \{1,\ldots, m-1\}$, the point $(\frac{s_{min}}{2},\frac{b_i\,+\,b_{i+1}}{2})$ is an interior point of $R$; and the outer segments of $R$ contain
the points $b_i$ and $b_{i+1}$.  This information suffices to determine the type of $R$, and to compute the limit values of the $F$-function
if necessary. The two infinite regions corresponding to $(-\infty,b_1)$ and $(b_m,\infty)$
can be ignored, because these regions must be of type (I) or type (II).
The same construction provides interior points of outer regions that are unbounded
in the $s$-direction, considering the vertical line at $s=2s_{max}$.

With this variant, we compute the matching distance with space complexity $O(n^4)$. 
The time complexity
remains $O(n^{11})$ as in the original algorithm, because we still iterate over $O(n^8)$ vertices and compute the bottleneck distance
in each position in $O(n^3)$ time.

\fi

\end{document}